\documentclass{amsart}

\usepackage{amssymb,amsmath,enumerate}

\usepackage[all]{xy}
\SelectTips{cm}{}

\newcommand{\nat}[1]{{#1}^{\natural}}
\newcommand\bse{{\boldsymbol e}}
\newcommand\bsf{{\boldsymbol f}}

\newcommand\BZ{{\mathbb Z}}
\newcommand\Coker{\operatorname{Coker}}
\newcommand{\cone}{\operatorname{Cone}}

\newcommand{\depth}{\operatorname{depth}}
\newcommand\edim{\operatorname{edim}}

\newcommand{\Tor}[4]{\operatorname{Tor}_{#1}^{#2}(#3,#4)}
\newcommand\fm{{\mathfrak m}}
\newcommand\fn{{\mathfrak n}}
\newcommand\fp{{\mathfrak p}}

\newcommand{\hh}[1]{\operatorname{H}(#1)}
\newcommand{\HH}[2]{\operatorname{H}_{#1}(#2)}
\newcommand\idmap{\operatorname{id}}
\newcommand\Image{\operatorname{Im}}
\newcommand\Ker{\operatorname{Ker}}

\newcommand{\ges}{\geqslant}
\newcommand{\lra}{\longrightarrow}
\newcommand{\wh}{\widehat}
\newcommand{\wt}{\widetilde}
\newcommand{\pdim}{\operatorname{proj\,dim}}
\newcommand{\rank}{\operatorname{rank}}
\newcommand{\shift}{{\sf\Sigma}}

\newcommand{\xla}{\xleftarrow}
\newcommand{\xra}{\xrightarrow}

\newcommand{\poin}[2]{\operatorname{P}^{#1}_{#2}(t)}

\newtheorem{theorem}{Theorem}[section]

\newtheorem{proposition}[theorem]{Proposition}
\newtheorem{lemma}[theorem]{Lemma}
\newtheorem{corollary}[theorem]{Corollary}

\theoremstyle{definition}
\newtheorem{example}[theorem]{Example}
\newtheorem{chunk}[theorem]{}

\theoremstyle{remark}
\newtheorem{remark}[theorem]{Remark}
{}
{}

\newtheorem*{Remark}{Remark}

\numberwithin{equation}{theorem}

\begin{document}

\title[Homology over trivial extensions]{Homology over trivial extensions of \\ commutative DG algebras}

\author[L.~L.~Avramov]{Luchezar L.~Avramov}
\address{Department of Mathematics,
University of Nebraska, Lincoln, NE 68588, U.S.A.}
\email{avramov@math.unl.edu}

\author[S.~B.~Iyengar]{Srikanth B.~Iyengar}
\address{Department of Mathematics,
University of Utah, Salt Lake City, UT 84112, U.S.A.}
\email{iyengar@math.utah.edu}

\author[S.~Nasseh]{Saeed Nasseh}
\address{Department of Mathematical Sciences,
Georgia Southern University, Statesboro, GA 30460, U.S.A.}
\email{snasseh@georgiasouthern.edu}

\author[S.~Sather-Wagstaff]{Sean Sather-Wagstaff}
\address{Department of Mathematical Sciences, O-110 Martin Hall, Box 340975,
Clemson, S.C. 29634, U.S.A.}
\email{ssather@clemson.edu}

\thanks{Avramov and Iyengar were partly supported by NSF grants DMS-1103176 (LLA) and DMS-1503044 (SBI).
Sather-Wagstaff was supported in part by a grant from the NSA}

\date{\today}


\begin{abstract} 
Conditions on the Koszul complex of a noetherian local ring $R$ guarantee that $\mathrm{Tor}^{R}_{i}(M,N)$ is 
non-zero for infinitely many $i$, when $M$ and $N$ are finitely generated $R$-modules of infinite projective 
dimension.  These conditions are obtained from results concerning Tor of differential graded modules over 
certain trivial extensions of commutative differential graded algebras.
\end{abstract}

\maketitle
   
\section*{Introduction}
This paper is motivated by, and feeds into our work in \cite{AINW}, which is concerned with the following problem: Given a 
commutative  noetherian ring $R$ and a finitely generated $R$-module $M$, does $\Tor iRMM=0$ for $i\gg 0$ imply that 
the projective dimension of $M$ is finite? Similar questions have arisen in the literature, also in certain non-commutative 
contexts; we refer the reader to \cite{AINW} for a  discussion.

When $R$ is complete intersection, by using their theory of cohomological support varieties Avramov and Buchweitz~\cite{AB}
answered that question in the positive and showed the failure in codimension two or higher of the following stronger property:
  \begin{equation}
    \label{eq:star}
      \tag{$*$}
\Tor iRMN=0 \text{ for $i\gg 0$ implies $\pdim_RM<\infty$ or $\pdim_RN<\infty$}\,. 
\end{equation} 

On the other hand, work of Huneke and Wiegand \cite{HW} and Jorgensen \cite{Jo} shows that $(*)$ does hold for Golod rings.  
More recently, Nasseh and Yoshino~\cite{NY} proved it for local rings whose maximal ideal requires a generator from the
socle.  Such rings are trivial extensions of the form $S\ltimes W$, where $S$ is a local ring and $W$ is a non-zero finitely 
generated $S$-module, annihilated by the maximal ideal of $S$.  

Even when a local ring is not a trivial extension, its Koszul complex---viewed as a differential graded (DG) algebra---may have such a structure. The goal of this paper is to prove that then the implication ($*$) still holds.  This is achieved in Theorem \ref{th:local}, which is deduced from much more general results concerning non-vanishing of Tor of DG modules over certain trivial extensions of DG algebras.

The substance of the paper is the development of techniques needed to state and prove this result; see Theorems~\ref{thm:dgfriendly}
and \ref{thm:dgfriendly2}, which in Proposition \ref{prop:Golod} give unified proofs of the results in \cite{HW, Jo, NY}.  Along the way, in 
Theorem~\ref{th:cdgretracts}, we obtain for retracts of augmented DG algebras a result that implies Herzog's \cite{He} computation of 
Poincar\'e series of modules over retracts of local rings; see Proposition \ref{prop:Herzog}. 

\section{Retracts of DG algebras}

In this section we establish statements concerning $\operatorname{Tor}$ functors of differential graded (DG) modules over retracts of 
DG algebras. Some basic definitions and constructions concerning DG algebras and their DG modules are recapped in an appendix
to the paper, to which there are frequent references throughout the text.

In the following paragraphs, we often consider bimodules: When $B,C$ are DG algebras, by a DG $BC$-bimodule we mean a complex
of abelian groups with compatible structures of a left DG $B$-module and a right DG $C$-module.  
 
\begin{chunk}
Let $\beta\colon B\to C$ be a morphism of DG algebras and $M$ a left DG $C$-module. We write ${}^{\beta}M$ for $M$ viewed as a left DG $B$-module by restriction of scalars along $\beta$; similarly for right DG modules. It is a routine verification that the maps
\begin{equation}
\label{eq:product-retract}
  \begin{aligned}
\iota^M\colon {}^{\beta}M &\lra {}^{\beta}(C^{\beta}\otimes_{B} {}^{\beta}M) \\
\iota^{M}(m)&\ =\ 1\otimes m
\end{aligned}
 \qquad \text{and} \qquad
\begin{aligned}
\mu^{M}\colon  C^{\beta}\otimes_{B} {}^{\beta}M &\lra M\\
 \mu^{M}(c\otimes m) &\ =\  cm
\end{aligned}
\end{equation}
are morphism of left DG $B$-modules and DG $C$-modules, respectively.  Note that 
   \begin{align}
     \label{eq:decomposition_iota}
\text{the composed map }
&{}^{\beta}M \xra{\ \cong\ } B\otimes_{B}{}^{\beta}M \xra{\ \beta\otimes_{B}{}^{\beta}M\ }  {}^{\beta}(C^{\beta}\otimes_{B} {}^{\beta}M)
\text{ is }\iota^M
  \\
     \label{eq:composition_iota}
\text{ the composed map }
&\mu^{M}\circ \iota^{M}
\text{ is the identity map of } M\,.
  \end{align}
\end{chunk}

\begin{lemma}
\label{le:dgretracts}
When $A\xra{\alpha}B\xra{\beta} C$ are morphisms of DG algebras, $L$ a right DG $A$-module, and $M$ a left DG 
$C$-module, there is an isomorphism of complexes
\[
 (L\otimes_A {}^{\beta\alpha}C)^{\beta}\otimes_B {}^{\beta}M 
 \cong(L\otimes_A {}^{\beta\alpha}M) \oplus\big(L\otimes_A {}^{\alpha}(\Coker(\beta)\otimes_B {}^{\beta}M)\big)\,.
\]
\end{lemma}

\begin{proof}
Consider the exact sequence of DG $BB$-bimodules:
\[
B \xra{\ \beta\ } C \lra \Coker(\beta) \lra 0\,.
\]
Applying $(? \otimes_{B}{}^{\beta}M)$ to it, in view of \eqref{eq:decomposition_iota} we get a sequence of left DG $B$-modules
\[
0\lra {}^{\beta}M\xra{\ \iota^M\ } {}^{\beta}(C^{\beta}\otimes_{B}{}^{\beta}M) \lra \Coker(\beta) \otimes_{B}{}^{\beta}M\lra 0\,.
\]
Its exactness is clear except  at $M$, and \eqref{eq:composition_iota} shows that $\iota^M$ is a split monomorphism. 
Thus, by restriction along $\alpha$, one gets an isomorphism of left DG $A$-modules
\[
{}^{\beta\alpha}(C^{\beta}\otimes_{B} {}^{\beta}M) \cong {}^{\beta\alpha}M \oplus {}^{\alpha}(\Coker(\beta) \otimes_{B}{}^{\beta}M)\,.
\]
The desired result is obtained by applying $(L\otimes_{A}?)$, then invoking the canonical isomorphism 
$L\otimes_A {}^{\beta\alpha}(C^{\beta}\otimes_B {}^{\beta}M)\cong (L\otimes_A {}^{\beta\alpha}C)^{\beta}\otimes_B {}^{\beta}M$.  
\end{proof}

\begin{chunk}
\label{ch:product-quism}
Let $\beta\colon B\to C$ be a morphism of DG algebras and $M$ a left DG $C$-module.  

When $\beta$ is a quasi-isomorphism and either $C^{\beta}$ or ${}^{\beta}M$ is semiflat, the morphisms of left 
DG modules $\iota^{M}$ and $\mu^{M}$, defined in \eqref{eq:product-retract} are quasi-isomorphisms.

Indeed,  $\beta\otimes_{B}{}^{\beta}M$ is a quasi-isomorphism by \ref{ch:semiflat}, so \eqref{eq:decomposition_iota} shows that
$\iota^{M}$ is a quasi-isomorphism, and then \eqref{eq:composition_iota} implies that so is $\mu^{M}$.
\end{chunk}

\begin{proposition}
\label{pr:dgretracts}
Let $A\xra{\alpha}B\xra{\beta} C$ be morphisms of DG algebras, $L$ a right DG  $C$-module, and $M$ a 
semiflat left DG $C$-module such that ${}^{\beta}M$ is semiflat. 

If $\beta\alpha$ is a quasi-isomorphism and the DG module $L^{\beta\alpha}$ or ${}^{\beta\alpha}C$ is semiflat, 
then there is a quasi-isomorphism of complexes
\[
L^{\beta}\otimes_{B}{}^{\beta}M\simeq\big(L\otimes_{C}M)\oplus(L^{\beta\alpha}\otimes_{A}{}^{\alpha}(\Coker(\beta)\otimes_{B}{}^{\beta}M)\big)\,.
\]
\end{proposition}

\begin{proof}
By (the analogue for right DG modules of) \ref{ch:product-quism} the map
\[
L^{\beta\alpha}\otimes_{A} {}^{\beta\alpha}C\lra  L
\]
is a  quasi-isomorphism of right DG $C$-modules and thus also one of right DG  $B$-modules. This, and the hypotheses on $M$, give the first and the last quasi-isomorphisms of complexes in the following string
\begin{alignat*}{2}
L^{\beta}\otimes_{B}{}^{\beta}M &\xla{\simeq}& 
       	  & (L^{\beta\alpha}\otimes_{A} {}^{\beta\alpha}C)^{\beta} \otimes_{B} {}^{\beta}M \\
&\cong& & (L^{\beta\alpha}\otimes_{A}{}^{\beta\alpha}M) \oplus (L^{\beta\alpha}\otimes_{A} {}^{\alpha}(\Coker(\beta)\otimes_{B} {}^{\beta}M)) \\
&\cong& & ((L^{\beta\alpha}\otimes_{A} {}^{\beta\alpha}C)\otimes_{C} M) \oplus (L^{\beta\alpha}\otimes_{A} {}^{\alpha}(\Coker(\beta)\otimes_{B} {}^{\beta}M))\\
&\xra{\simeq}& & (L\otimes_{C} M) \oplus (L^{\beta\alpha}\otimes_{A} {}^{\alpha}(\Coker(\beta)\otimes_{B} {}^{\beta}M))
\end{alignat*}
The second one is Lemma~\ref{le:dgretracts}, applied to $L^{\beta\alpha}$; the third one is canonical.
\end{proof}

Here is a first application of  Proposition~\ref{pr:dgretracts}. Note that the DG algebras in the statement are graded-commutative.

\begin{theorem}
\label{th:cdgretracts}
Let $ B\xra{\beta} C\xra{\varepsilon} k$ be morphisms of graded-commutative DG algebras, where $k$ is a field, and let $L$ be a DG $C$-module.  

If there exists a morphism of DG algebras $\alpha\colon A\to B$, such that $\beta\alpha\colon A\to C$ is a quasi-isomorphism, 
then there is an isomorphism of graded $k$-vector spaces:
\[
\Tor {}B{L^{\beta}}{{}^{\varepsilon\beta}k} \cong \Tor{}CL{{}^{\varepsilon}k} \otimes_{k} \Tor {}B{C^{\beta}}{{}^{\varepsilon\beta}k} \,.
\]
\end{theorem}

\begin{proof}
Referring to \ref{ch:resolutions-map}, form a commutative diagram of DG algebras
\[
\xymatrix{
                    &\wt B \ar@{->>}[d]^{\simeq} \ar@{>->}[r]^{\wt \beta}  & \wt C \ar@{->>}[d]^{\simeq} \\
A\ar@{->}[r]^{\alpha} \ar@{>->}[ur]^{\wt \alpha} &B \ar@{->}[r]^{\beta}  & C  \ar@{->}[r]^{\varepsilon} & k
  }
\]
where $A\xra{\wt\alpha} \wt B \overset{\simeq}{\twoheadrightarrow} B$ is a semiflat DG algebra resolution of $\alpha$ and $\wt B\xra{\wt\beta} \wt C\overset{\simeq}{\twoheadrightarrow} C$ is one of the composed morphism $\wt B\to B\xra{\beta}C$.  In view of \ref{ch:tor2}, it suffices to establish the desired isomorphism for the morphism of DG algebras $\wt B\to \wt C\to k$. Thus, replacing $B\to C\to k$ by $\wt B\to \wt C\to k$ we may assume that  ${}^{\alpha} B$ and ${}^{\beta}C$ are semiflat.  Moreover, replacing $L$ with a resolution, we may further assume that $L$ is semiflat. Note that ${}^{\beta\alpha}C$ and $L^{\beta\alpha}$ are semiflat, by \ref{ch:flat-maps}.

One has an exact sequence of DG $B$-modules
\[
0\lra B \xra{\ \beta\ } C \lra \Coker(\beta) \lra 0\,.
\]
Applying $\Tor{}B{?}{{}^{\varepsilon\beta}k}$ one gets an isomorphism of graded $k$-vector spaces
\begin{equation}
\label{eq:decomposition}
k\oplus \Tor{}B{\Coker(\beta)}{{}^{\varepsilon\beta}k} \cong \Tor{}B{C^{\beta}}{{}^{\varepsilon\beta}k}\,.
\end{equation}

Let $M\xra{\simeq}{}^{\varepsilon}k$ be a semiflat resolution over $C$. Since $\Coker(\beta)$ is semiflat, by construction, it induces a quasi-isomorphism of DG $B$-modules
\[
\Coker(\beta)\otimes_{B} {}^{\beta}M\xra{\ \simeq\ } \Coker(\beta)\otimes_{B}{}^{\varepsilon\beta}k\,.
\]
By restriction of scalars, this is also a morphism of DG $A$-modules. Since $L^{\beta\alpha}$ is semiflat, the preceding quasi-isomorphism 
induces the one below:
\begin{align*}
L^{\beta\alpha}\otimes_A {}^{\alpha}(\Coker(\beta)\otimes_{B}{}^{\beta}M) 
	& \simeq L^{\beta\alpha}\otimes_A {}^{\alpha}(\Coker(\beta)\otimes_{B}{}^{\varepsilon\beta}k) \\
	& \cong (L^{\beta\alpha}\otimes_{A} {}^{\varepsilon\beta\alpha}k)\otimes_{k} (\Coker(\beta)\otimes_{B}{}^{\varepsilon\beta}k) 
\end{align*}
The isomorphism holds because the action of $B$ on $\Coker(\beta)\otimes_{B}k$ through $\Coker(\beta)$ coincides with is action through $k$, and hence so do the induced actions of $A$.  

The quasi-isomorphisms above and the K\"unneth formula yield the first one of the following isomorphisms of graded $k$-vector spaces:
\begin{align*}
\hh{L^{\beta\alpha}\otimes_A {}^{\alpha}(\Coker(\beta)\otimes_{B}{}^{\beta}M)} 
	&\cong \Tor{}A{L^{\beta\alpha}}{{}^{\varepsilon\beta\alpha}k} \otimes_{k} \Tor {}B{\Coker(\beta)}{k^{\varepsilon\beta}}\\
	&\cong \Tor{}CL{{}^{\varepsilon}k} \otimes_{k} \Tor {}B{\Coker(\beta)}{k^{\varepsilon\beta}}
\end{align*}
The second one holds by \ref{ch:tor2}, since  $\beta\alpha$ is a quasi-isomorphism. 

The last display justifies the third isomorphism in the next string:
\begin{align*}
\Tor{}B{L^{\beta}}{k^{\varepsilon\beta}} &\cong \hh{L^{\beta}\otimes_{B}{}^{\beta}M}\\
	&\cong \hh{L\otimes_{C}M} \oplus \hh{L^{\beta\alpha}\otimes_A {}^{\alpha}(\Coker(\beta)\otimes_{B}{}^{\beta}M)} \\
	&\cong \Tor{}CL{{}^{\varepsilon}k}\oplus \big(\Tor{}CL{{}^{\varepsilon}k} \otimes_{k}\Tor{}B{\Coker(\beta)}{k^{\varepsilon\beta}}\big) \\
	&\cong \Tor{}CL{{}^{\varepsilon}k}\otimes_{k}\big(k\oplus \Tor{}{B}{\Coker(\beta)}{k^{\varepsilon\beta}}\big) \\
	&\cong \Tor{}CL{{}^{\varepsilon}k}\otimes_{k}\Tor{} B{C^{\beta}}{k^{\varepsilon\beta}} 
\end{align*}
Proposition~\ref{pr:dgretracts} gives the second isomorphism, and formula \eqref{eq:decomposition} the last one.
\end{proof}

\section{Trivial extensions}

\emph{For the rest of the article all DG algebras are assumed to be graded-commutative}. 
 
  \medskip

Let $A$ be a DG algebra and $W$ a DG $A$-module.

The \emph{trivial extension} $A\ltimes W$ is the DG algebra with underlying complex  $A\oplus W$
and product given by $(a,w)(a',w')=(aa',aw'+(-1)^{|w||a'|}a'w)$. Note that the canonical maps 
$A\to A\ltimes W\to A$ are morphisms of DG algebras.

\begin{theorem}
\label{thm:tor}
Let  $A$ be a DG algebra, and let $M$ and $N$ be DG $A$-modules.

Let $k$ be a field, $W$ a DG $k$-module, and $\varepsilon\colon A\to k$ a morphism of DG algebras.

Set $B=A\ltimes {}^{\varepsilon}W$ and let $\beta\colon B\to A$ be the canonical surjection.

There is then a natural isomorphism of graded $\hh A$-modules:
\[
\Tor {}B{M^{\beta}}{{}^{\beta}N} 
\cong \Tor {}AMN \oplus \big(\Tor {}AM{{}^{\varepsilon}k} \otimes_{k} (\shift \hh W) \otimes_{k}\Tor {}B{k^{\varepsilon\beta}}{{}^{\beta}N}\big)\,.
\]
  \end{theorem}

\begin{corollary}
\label{co:tor-nonvanishing}
When $\HH iV\ne0$ holds for some $i\ne -1$, the condition 
  \[
\Tor{}AM{{}^{\varepsilon}k}\ne0\ne\Tor{}A{k^{\varepsilon}}N
  \]
implies $\Tor {i}B{M^{\beta}}{{}^{\beta}N}\ne 0$ for infinitely many integers $i$.
\end{corollary}

In the proofs we use basic properties of mapping cones, which we recall next.

\begin{chunk}
\label{ch:cone-ker}
Let $\psi\colon S\to T$ be a morphism of DG $A$-modules.

The \emph{cone} of $\psi$ is the DG $A$-module $\cone(\psi)$, with $\nat{\cone(\psi)}=\shift \nat S\oplus \nat T$ 
and differential given by $(s,t) \mapsto (\partial^{\shift S}(s), \partial^{T}(t) + \psi(s))$.

If $\psi$ is injective, then there is a quasi-isomorphism of DG $A$-modules 
  \begin{equation}
    \label{eq:cone-coker}
\pi\colon \cone(\psi) \xra{\ \simeq } \Coker(\psi) \quad\text{given by}\quad (s,t) \mapsto t+\Image(\psi)\,.
   \end{equation}
Indeed, then $\pi$ is surjective with $\Ker(\iota)\cong\cone(\idmap^{S})$, and $\cone(\idmap^{S})$ is acyclic.

If $\psi$ is surjective, then there is a quasi-isomorphism of DG $A$-modules 
  \begin{equation}
    \label{eq:cone-ker}
\iota\colon \shift \Ker(\psi) \xra{\ \simeq } \cone(\psi) \quad\text{given by}\quad s\mapsto (s,0)\,.
      \end{equation}
Indeed, then $\iota$ is injective with $\Coker(\iota)\cong\cone(\idmap^{T})$, and $\cone(\idmap^{T})$ is acyclic.

If there is a commutative square of morphisms of DG $A$-modules
\[
\xymatrix{
S\ar@{->}[r]^{\psi} \ar@{->}[d]_{\sigma}^{\simeq} & T\ar@{->}[d]^{\tau}_{\simeq} \\
S'\ar@{->}[r]^{\psi'}   & T' 
}
\]
with $\sigma,\tau$ quasi-isomorphisms, then there is a quasi-isomorphism of DG $A$-modules 
  \begin{equation}
    \label{eq:cone-induced}
\underline\psi\colon \cone(\psi)\xra{\simeq} \cone(\psi') \quad\text{is given by}\quad (s,t) \mapsto (\sigma(s),\tau(t))\,.
      \end{equation}
Indeed, this follows from the Five-Lemma applied to the commutative diagram 
\[
\xymatrix{
0 \ar@{->}[r] &T\ar@{->}[r]\ar@{->}[d]_{\tau}^{\simeq} &\cone(\psi)\ar@{->}[r]\ar@{->}[d]_{\underline\psi} &\shift S\ar@{->}[r] \ar@{->}[d]_{\shift\sigma}^{\simeq} &0
  \\
0 \ar@{->}[r] & T' \ar@{->}[r] &\cone(\psi') \ar@{->}[r]  & \shift S' \ar@{->}[r]  &0
}
\]
\end{chunk}

\stepcounter{theorem}

 \begin{proof}[Proof of Theorem~\emph{\ref{thm:tor}}]
By using \ref{ch:resolutions-map}, we construct a diagram of DG algebras
\[
\xymatrix{
 A\ar@{->}[r]^{\alpha} \ar@{>->}[dr]_{\iota} &\wt B \ar@{->>}[d]^{\simeq}_{\wt\iota} \ar@{>->}[r]^{\wt\beta}  & C \ar@{->>}[d]_{\simeq}^{\gamma}\\
                   &B \ar@{->>}[r]^{\beta}  & A \ar@{->}[r]^{\varepsilon} & k }
\]
where $\iota$ is canonical, $\wt\iota\alpha$ is a semiflat resolution of $\iota$, and 
$\gamma\wt\beta$ is one of $\beta\wt\iota$.

Let $\wt M\xra{\simeq}M^{\gamma}$ and $\wt N\xra{\simeq}{}^{\gamma}N$ be semiflat resolutions over $C$.
In view of \ref{ch:flat-maps}, the maps $\wt M^{\wt\beta}\xra{\simeq}M^{\gamma\wt\beta}=M^{\beta\wt\iota}$ 
and ${}^{\wt\beta}\wt N\xra{\simeq}{}^{\gamma\wt\beta}N={}^{\beta\wt\iota}N$ are semiflat resolutions over $\wt B$. 
They explain the first isomorphisms below, and \ref{ch:tor2} gives the second ones:
  \begin{gather*}
\hh{\wt M^{\wt\beta}\otimes_{\wt B}{}^{\wt\beta}\wt N}
\cong\Tor{}{\wt B}{M^{\beta\wt\iota}}{{}^{\beta\wt\iota}N}
\cong \Tor{}B{M^{\beta}}{{}^{\beta} N}
  \\
\hh{\wt M\otimes_{C}\wt N}
\cong \Tor{}C{M^{\gamma}}{{}^{\gamma}N}
\cong \Tor{}AMN\,.
  \end{gather*}
In view of these isomorphisms, Proposition \ref{pr:dgretracts} applied with $A\xra{\iota}\wt B\xra{\wt\beta} C$ yields
\begin{equation}
   \label{eq:tor1}
\Tor{}B{M^{\beta}}{{}^{\beta} N} 
	\cong \Tor{}AMN \oplus \hh{{\wt M}^{\wt\beta\alpha}\otimes_{A}{}^{\alpha}(\Coker(\wt\beta)\otimes_{\wt B}{}^{\wt\beta}\wt N)}\,.
\end{equation}
The rest of the argument goes into computing the homology on the right hand side.

Since $\wt\beta$ is injective and $\beta$ is surjective, \eqref{eq:cone-coker} and \eqref{eq:cone-ker} give quasi-isomorphisms
$\Coker(\wt\beta) \simeq \cone(\wt\beta)$ and $\cone(\beta) \simeq \shift{}W^{\varepsilon\beta}$, respectively.  From 
\eqref{eq:cone-induced} we further obtain $\cone(\wt\beta)\simeq{}^{\wt\iota}\cone(\beta)$, so we get a quasi-isomorphism of DG ${\wt B}$-modules
  \begin{equation*}
\Coker(\wt\beta) \simeq \shift{}W^{\varepsilon\beta\wt\iota}=\shift{}W^{\varepsilon\gamma\wt\beta}\,.
  \end{equation*}
Since ${}^{\wt\beta}\wt N$ is semiflat, it induces a quasi-isomorphism of DG $A$-modules
\[
{}^{\alpha}(\Coker(\wt\beta)\otimes_{\wt B}{}^{\wt\beta}\wt N)
\simeq {}^{\alpha}(\shift W^{\varepsilon\gamma\wt\beta}\otimes_{\wt B}{}^{\wt\beta}\wt N)\,.
\]
As $\wt M^{\wt\beta\alpha}$ is semiflat, the preceding quasi-isomorphism induces the one in the display
\begin{align*}
\wt M^{\wt\beta\alpha}\otimes_{A}{}^{\alpha}(\Coker(\wt\beta)\otimes_{\wt B}\wt N)
	&\simeq \wt M^{\wt\beta\alpha}\otimes_{A}{}^{\alpha}(\shift W^{\varepsilon\gamma\wt\beta}\otimes_{\wt B}{}^{\wt\beta}\wt N) \\
	&\cong (\wt M^{\wt\beta\alpha}\otimes_{A} {}^{\varepsilon}k)\otimes_{k}(\shift W)\otimes_{k}(k^{\varepsilon\beta\wt\iota}\otimes_{\wt B}{}^{\wt\beta}\wt N)\,.
\end{align*}
From the semiflat resolution ${}^{\wt\beta}\wt N\xra{\simeq}{}^{\beta\wt\iota}N$ and \ref{ch:semiflat}, we get isomorphisms
  \[
\hh{k^{\varepsilon\beta\wt\iota}\otimes_{\wt B}{}^{\wt\beta}\wt N}
\cong\Tor{}{\wt B}{k^{\varepsilon\beta\wt\iota}}{{}^{\beta\wt\iota} N}
\cong \Tor{}B{k^{\varepsilon\beta}}{{}^{\beta} N}\,.
  \]
Finally, the semiflat resolution $\wt M^{\wt\beta\alpha}\xra{\simeq}M$ yields
  \[
\hh{\wt M^{\wt\beta\alpha}\otimes_{A} {}^{\varepsilon}k}
\cong \Tor{}AM{{}^{\varepsilon}k}\,.
  \]
The formulas in the last three displays and the K\"unneth isomorphism give 
 \begin{equation}
    \label{eq:tor2}
\hh{\wt M^{\wt\beta\alpha}\otimes_{A}{}^{\alpha}(\Coker(\beta)\otimes_{\wt B}\wt N)}
\cong \Tor{}AM{{}^{\varepsilon}k}\otimes_{k}\shift \hh W \otimes_{k}\Tor{}B{k^{\varepsilon\beta}}{{}^{\beta} N}\,.
  \end{equation}

Combining \eqref{eq:tor1} and \eqref{eq:tor2} yields the isomorphism in the statement of Theorem~\ref{thm:tor}.
It is natural, as it was obtained as a composition of natural morphisms. 
  \end{proof}

\stepcounter{theorem}

\begin{proof}[Proof of Corollary \emph{\ref{co:tor-nonvanishing}}]
To simplify notation, we let $k$ stand also for ${}^{\varepsilon}k$ and for ${}^{\varepsilon\beta}k$.

We have $\Tor{}AMk\ne0\ne\hh W$ by hypothesis, so by Theorem~\ref{thm:tor} it suffices to prove 
$\Tor iBk{{}^{\beta}N}\ne 0$ for infinitely many $i$.  From $\Tor{}Bk{{}^{\beta}N}\cong\Tor{}B{N^{\beta}}k$ and another reference to Theorem~\ref{thm:tor}, we see that it suffices to show $\Tor{i}Bkk\ne0$ for infinitely many $i$; that is, the validity of the following alternative:
  \begin{equation}
  \label{eq:supB}
\sup \Tor{}Bkk =\infty \quad\text{or} \quad \inf \Tor{}Bkk =-\infty\,.
  \end{equation}

We start by proving that there are inequalities
  \begin{equation}
  \label{eq:supA}
\sup \Tor{}Akk \ge 0 \quad\text{and} \quad \inf \Tor{}Akk \le 0\,.
  \end{equation}
Let $A\to \wt A\to k$ be a semiflat resolution of the DG $A$-algebra $k$; see \ref{ch:resolutions-map}. 
It induces the first two arrows in the next string, where the last one is multiplication:
\[
k=A\otimes_{A}k \lra \wt A\otimes_{A}k \lra k\otimes_{A}k \lra k\,.
\]
The composed map sends $1$ to $1$, so is the identity map of $k$. The induced maps 
$k\to \Tor{}Akk\to k$ also compose to $\operatorname{id}^k$.  We get $\Tor 0Akk\ne0$, so \eqref{eq:supA} holds.
 
Suppose, by way of contradiction, that \eqref{eq:supB} fails, so that $\sup \Tor{}Bkk$ and $\inf \Tor{}Bkk$ 
are both finite. The isomorphism of graded $k$-vector spaces
\begin{equation}
\label{eq:tor4}
\Tor{}Bkk \cong \Tor{}Akk \oplus \big(\Tor {}Akk \otimes_{k}\shift \hh W\otimes_{k} \Tor{}Bkk\big)\,,
\end{equation}
given by Theorem~\ref{thm:tor}, then implies that $\sup \Tor{}Akk$ and $\inf \Tor{}Akk$ are finite, ditto for
$\sup \hh W$ and $\inf \hh W$.

If $\inf\hh W\le -2$, then \eqref{eq:tor4}, and the corresponding estimates for $B$, imply
\[
\inf\Tor{}Bkk =  1 + \inf\Tor{}Akk + \inf \hh W + \inf \Tor{}Bkk\,,
\]
which contradicts $\inf\Tor{}Akk\le 0$.
We conclude that $\inf\hh W\ge -1$ holds. Then $\sup\hh W\ge 0$, by the hypothesis on $W$. Again from \eqref{eq:tor4} one gets 
\[
\sup\Tor{}Bkk =  1 + \sup\Tor{}Akk + \sup \hh W + \sup \Tor{}Bkk\,.
\]
Once again, this is impossible, this time because $\sup\Tor{}Akk\ge 0$.

This gives the desired contradiction, and completes the proof of the corollary.
   \end{proof}

The next example shows that in Corollary~\ref{co:tor-nonvanishing} the hypothesis on $W$ is necessary.

\begin{example}
When $k$ is a field and $W=\Sigma^{-1}k$, one has 
\[
\Tor i{k\ltimes W}kk \cong 
\begin{cases}
k\langle x\rangle & \text{for $i=0$} \\
0       & \text{for $i\ne 0$}
\end{cases}
\]
where $k\langle x\rangle$ denotes a divided powers algebra on an indeterminate $x$.
\end{example}

\section{Local DG algebras}

In this section, as in the preceding one, we consider DG modules over a DG algebra $B$ quasi-isomorphic to $A\ltimes W$ when $A$ is augmented to a field, $k$, and $W$ is a DG $k$-module.  The goal here is to prove that the boundedness of $\Tor{}BMN$ for DG $B$-modules $M$ and $N$ implies strong structural restrictions on $M$ or $N$. In order to do this, we need additional hypotheses on $A$.

\begin{chunk}
\label{ch:localDGA}
In this paper we say that $(A,\fm,k)$ is a \emph{local DG algebra} if the following hold:
\begin{enumerate}[{\quad\rm(a)}]
\item
$A$ is a DG algebra with $A_{<0}=0$, and $\fm_{i}=A_{i}$ for $i\ne 0$.
\item
$A_{0}$ is a noetherian ring with unique maximal ideal $\fm_0$, and $k=A_0/\fm_0$.
\item
$\hh A$ is degreewise finite.
\item
$\HH 0A$ is not equal to $0$.
\end{enumerate}
In particular, $\fm$ is a DG ideal, called the \emph{maximal ideal} of $A$, and the natural map $\varepsilon\colon A\to k$ is a morphism of DG algebras, called the \emph{canonical augmentation}. 
 \end{chunk}

The notion of perfect DG module, used in the next result, is defined in \ref{ch:semifree}.

\begin{theorem}
\label{thm:dgfriendly}
Let $B$ be a DG algebra that is quasi-isomorphic to $A\ltimes {}^{\varepsilon}W$, where $(A,\fm,k)$ is a local DG algebra with $\hh A$ bounded, and $W$ is a DG $k$-module with $\hh W$ finite and $\HH {< 0}W=0\ne \hh W$.

If $M$ and $N$ are DG $B$-modules, such that $\hh M$ and $\hh N$ are bounded and degreewise finite and $\Tor{}BMN$ 
is bounded, then $M$ or $N$ is perfect.
\end{theorem}

The proof utilizes minimal semifree resolutions, which we review next.

\begin{chunk}
\label{ch:minimal}
Let $(B,\fn,k)$ be a local DG algebra and $M$ a DG $B$-module, such that $\HH{}M$ is bounded below and is degreewise finite.

The DG module $M$ admits a \emph{minimal semifree resolution};  that is, a quasi-isomor\-phism 
$E\xra{\simeq}M$, where $E$ is semifree and $\partial(E)\subseteq \fn E$; see, for example, \cite{AFH}.  

Any minimal semifree resolution $E\xra{\simeq}M$ has $\inf E=\inf\hh M$, every basis $\bse$ of the graded 
$\nat B$-module $\nat E$ is degreewise finite, and for $i\in\BZ$ one has
  \begin{equation}
    \label{eq:minimalL}
\Tor{i}BMk\cong\HH{i}{E\otimes_Bk}=(E\otimes_Bk)_{i}\cong\bigoplus_{e\in\bse,\, |e|=i}ke\,.
   \end{equation}
\end{chunk}

\begin{proposition}
\label{pr:syzygy}
Let $(B,\fn,k)$ be a local DG algebra and $M$ a DG $B$-module with $\hh M$ bounded and degreewise finite.

There exists an exact sequence of DG $B$-modules
\[
0\lra M' \lra F \lra M''\lra 0
\]
with $F$ finite semifree, $M'\subseteq \fn F$, and $M''\simeq M$ with $\inf M''=\inf\hh M$.
\end{proposition}

\begin{proof}
After replacing $M$ with a minimal semifree resolution, we may assume that $M$ has a semibasis $\bse$ 
and satisfies $\partial (M)\subseteq \fn M$.  Setting $\bsf=\{p\in\bse: |e|\le s\}$, where $s=\sup\hh M$, and $F=B\bsf$, note that
$\bsf$ is a semibasis of $F$, it is finite by \ref{ch:minimal}, and $\partial(F)\subseteq\fn F$ holds.

The subset $L=M_{\ges s+1}\cup \partial(M_{s+1})$ is a DG $B$-submodule of $M$ with $\HH{}{L}=0$.  
Thus, $M''=M/L$ has $M''_i=0$ for $i\ge s+1$, and the natural map $M\to M''$ is a surjective quasi-isomorphism of DG $B$-modules.

The composed map $F\hookrightarrow M\twoheadrightarrow M''$ is a surjective morphism of DG $B$-modules. 
Let $M'$ denote its kernel.  By construction one then has
  \[
M'_i
=\begin{cases}
0 &\text{for }i\le s-1\,;\\
\partial(F_{s+1}) &\text{for }i=s\,;\\
F_i=\sum_{h=1}^s B_hF_{i-h} &\text{for }i\ge s+1\,.
\end{cases}
  \]
In particular, $M'\subseteq \fn F$. Thus, the DG modules $M'$, $F$ and $M''$ yield the desired exact sequence.
  \end{proof}

\stepcounter{theorem}

\begin{proof}[Proof of Theorem~\emph{\ref{thm:dgfriendly}}]
As $k$ is a field, we can choose a quasi-isomorphism $W\simeq\hh W$ of DG $k$-modules.  It yields one between the DG $A$-modules ${}^{\varepsilon}W$ and ${}^{\varepsilon}\hh W$ and hence a quasi-isomorphism $A\ltimes {}^{\varepsilon}W\simeq A\ltimes {}^{\varepsilon}\hh W$ of DG algebras.   Thus, we obtain a composite quasi-isomorphism $B\simeq A\ltimes {}^{\varepsilon}\hh W$ of DG algebras.  

In view of \ref{ch:equivalence}, it suffices to prove the theorem for $B=A\ltimes {}^{\varepsilon}W$, where $W$ is a nonzero finite DG $k$-module with $V_{<0}=0$ and $\partial(W)=0$. Then, since $\hh A$ is bounded, the same is true of $\hh B$, and hence any semifree DG $B$-module is degreewise finite and bounded.

As $(B,\fn,k)$ with $\fn=B(\fm, {}^{\varepsilon}W)$ is local, Proposition~\ref{pr:syzygy} gives 
finite semifree DG $B$-modules $F$ and $G$ and exact sequences of DG $B$-modules
   \begin{gather}
     \label{eq:dgfriendly1L}
0\lra M' \lra F \lra M''\lra 0
  \\
     \label{eq:dgfriendly2L}
0\lra N' \lra G \lra N''\lra 0
  \end{gather}
where $M'\subseteq \fn F$ and $N'\subseteq \fn G$ hold, and $M''$ and $N''$ are bounded and quasi-isomorphic 
to $M$ and $N$, respectively.  In particular, for $i\gg0$ we have
   \begin{gather}
     \label{eq:dgfriendly3L}
\Tor{i}{B}F{N''}\cong\HH i{F\otimes_{B}N''}=0
  \\
     \label{eq:dgfriendly4L}
\Tor{i}{B}{M''}G\cong\HH i{M''\otimes_{B}G}=0
  \\
     \label{eq:dgfriendly5L}
\Tor{i}{B}{M''}{N''}\cong\Tor{i}{B}{M}{N}=0
  \end{gather}

Due to \eqref{eq:dgfriendly3L} and \eqref{eq:dgfriendly5L}, the exact sequence \eqref{eq:dgfriendly1L}
yields $\Tor{i}{B}{M'}{N''}=0$ for $i\gg0$.  By using the latter equalities and \eqref{eq:dgfriendly4L}, from
the exact sequence \eqref{eq:dgfriendly2L} we obtain $\Tor{i}{B}{M'}{N'}=0$ for $i\gg0$.  In addition,
$\Tor{i}{B}{M'}{N'}=0$ holds for $i\ll0$, as $M$ and $N$ are homologically bounded.  The DG module $M'$ and $N'$ satisfy
  \[
({}^{\varepsilon}W) M'\subseteq({}^{\varepsilon}W)\fn F=0=({}^{\varepsilon}W)\fn G\supseteq({}^{\varepsilon}W) N'\,,
  \]
so we have $M'={}^{\beta\alpha}M'$ and $N'={}^{\beta\alpha}N'$, where $A\xra{\alpha}B\xra{\beta}A$ are the natural maps. 

Now Corollary~\ref{co:tor-nonvanishing} gives $\Tor{}{A}{{}^{\alpha}M'}k=0$ or $\Tor{}{A}k{{}^{\alpha}N'}=0$.  In view of
\eqref{eq:minimalL}, this means that $0\xra{\simeq}{}^{\alpha}M'$ or $0\xra{\simeq}{}^{\alpha}N'$ is a minimal free 
resolution.  Thus, we have $\hh{M'}=0$, and then \eqref{eq:dgfriendly1L} gives $F\simeq M''\simeq M$, or $\hh{N'}=0$, and
then \eqref{eq:dgfriendly2L} gives $G\simeq N''\simeq N$.  We have proved that $M$ or $N$ is perfect, as desired.
  \end{proof}

\begin{remark}
  \label{re:PS}
Let $C$ be a local DG algebra with residue field $k$, and let $L$ be a DG $C$-module with $\hh L$ degreewise finite 
and bounded below.  The graded vector space $\Tor {}CLk$ then has the same properties, see \ref{ch:minimal}, 
so a formal Laurent series
\[
\poin CL = \sum_{i\in \BZ}\rank_{k}(\Tor iCLk) t^{i}\in\BZ[[t]][t^{-1}]
\]
is defined. It is known as the \emph{Poincar\'e series} of $L$ over $C$.

Let $B$ be a local DG algebra with residue field $k$ and $\beta\colon B\to C$ a morphisms of local DG 
algebras commuting with the canonical augmentations.  If there is a morphism of DG algebras 
$\alpha\colon A\to B$, such that $\beta\alpha$ is a quasi-isomorphism, then 
  \begin{equation}
    \label{eq:PS}
\poin BL = \poin BC\poin CL
  \end{equation}
holds in $\BZ[[t]][t^{-1}]$, due to the isomorphism in Theorem~\ref{th:cdgretracts}.  

This formula holds, in particular, when $C$ is a DG algebra retract of $B$.
\end{remark}

\section{Koszul extensions}

Here we widen the range of applications of Theorem \ref{thm:dgfriendly} by weakening some 
of its hypotheses, by means of the classical construction of adjunction of indeterminates.

  \begin{chunk}
    \label{ch:koszul}
Let $B$ be a commutative DG algebra and $z$ a cycle with $|z|$ even.  

A DG algebra $B_z\langle x\rangle$ is defined by $\nat{B_z\langle x\rangle}=\nat B\otimes_{\BZ}\BZ\langle x\rangle$, where 
$\BZ\langle x\rangle$ is the exterior algebra of a free $\BZ$-module $\BZ x$ with $|x|=|z|+1$, and 
  \[
\partial(b+cx) = \partial(b) + \partial(c)x + (-1)^{|c|}cz\,.
  \]
 
A \emph{Koszul extension} of $B$ is a DG algebra of the form $B\langle X\rangle$, where $X=x_{1},\dots,x_n$ is a sequence of indeterminates of odd degrees, and for $i=1,\dots,n$ there are cycles $z_i\in B\langle x_{1},\dots,x_{i-1}\rangle$, such that $B\langle x_{1},\dots,x_{i}\rangle=B_{z_i}\langle x_{1},\dots,x_{i-1}\rangle\langle x_{i}\rangle$.  

The  inclusion $B\subseteq B\langle X\rangle $ is a morphism of DG algebras. 

When $M$ is a DG $B$-module we let $M\langle X\rangle$ denote the $B\langle X\rangle $-module $B\langle X\rangle \otimes_{B} M$. 
  \end{chunk}

The terminology adopted above is a reminder that the Koszul complex on a sequence of elements 
$z_{1},\dots,z_{n}$ in a commutative ring $R$ is a Koszul extension of~$R$. 

\begin{theorem}
\label{thm:dgfriendly2}
Let $(B,\fn,k)$ be a local DG algebra.

Assume that some Koszul extension of $B$ is quasi-iso\-mor\-phic to $A\ltimes {}^{\varepsilon}W$, 
where $(A,\fm,k)$ is a local DG algebra with $\hh A$ bounded and $W$ is a DG $k$-module with $\hh W$ nonzero and bounded.

If $M$ and $N$ are DG $B$-modules, such that $\hh M$ and $\hh N$ are bounded and degreewise 
finite and $\Tor{}BMN$ is bounded, then $M$ or $N$ is perfect.
\end{theorem}

The next result collects standard properties of Koszul extensions needed in the proof of the 
preceding theorem; proofs are included for ease of reference. 

\begin{lemma}
\label{lem:exterior}
Let $B$ be a DG algebra and $B\langle X\rangle$ a Koszul extension of $B$. 

Let $M$ and $N$ be DG $B$-module.
\begin{enumerate}[{\quad\rm(1)}]
\item 
If $\hh M$ is bounded, then so $\hh{M\langle X\rangle}$.
\item 
If $\Tor {}BMN$ is bounded, then so is $\Tor {}{B\langle X\rangle }{M\langle X\rangle }{N\langle X\rangle}$.
\item  
If $N$ is a DG $B\langle X\rangle$-module, then $\Tor {}{B\langle X\rangle}{M\langle X\rangle}N\cong \Tor {}BMN$ holds.
\item 
If $\HH 0B$ is noetherian and $\hh M$ is degreewise finite, then 
$\HH 0{B\langle X\rangle}$ is noetherian and $\hh{M\langle X\rangle}$ is degreewise finite.
\item 
If $(B,\fn,k)$ is local and $B_0\cap\partial X\subseteq\fn$, then $(B\langle X\rangle,B\langle X\rangle(\fn,X),k)$ is local.
   \end{enumerate}
\end{lemma}

\begin{proof}
By induction, it suffices to treat the case $X=\{x\}$; set $|x|=d+1$. 

Applying $(?)\otimes_{B}M$ to the exact sequence of DG $B$-modules
\[
0\lra B \lra B\langle x\rangle  \lra xB\lra 0
\]
yields, in homology, an exact sequence of $\HH 0B$-modules 
\[
0\lra \HH i{M}/z\HH{i-d}M \lra\HH i{M\langle x\rangle } \lra (0\, \colon z)_{\HH{i-d-1}M}\lra 0
\]
for every $i\in\BZ$.  Parts (1) and (4) follow, and the latter implies part (5).

In the remainder of the proof we may assume that the DG $B$-module $M$ is semiflat.  The DG 
$B\langle X\rangle $-module $M\langle X\rangle$ then is semiflat, by \ref{ch:flat-maps}, so we have
\[
\Tor{}BM{?} \cong \hh{M\otimes_{B}?}\quad\text{and}\quad \Tor{}{B\langle X\rangle }{M\langle X\rangle }{?} \cong \hh{M\langle X\rangle \otimes_{B\langle X\rangle }?}\,.
\]
The definition of Koszul extensions gives an isomorphism 
\[
M\langle X\rangle \otimes_{B\langle X\rangle }N\langle X\rangle  \cong (M\otimes_{B}N)\langle X\rangle 
\]
of DG $B\langle X\rangle $-modules, which proves (2).  Part (3) follows from the isomorphisms 
\[
M\langle X\rangle \otimes_{B\langle X\rangle }N = (B\langle X\rangle  \otimes_{B}M)\otimes_{B\langle X\rangle }N \cong M\otimes_{B}N\,.
  \qedhere
\]
  \end{proof}

One advantage of local DG algebras is that perfection can be detected by homology.
This is the content of the next result, a variation on \cite[4.8 and 4.10]{ABIM}.

\begin{proposition}
\label{pr:perfect}
Let $(B,\fn,k)$ be a local DG algebra and $M$ a left DG $B$-module.

The following conditions are equivalent:
\begin{enumerate}[{\quad\rm(i)}]
\item
$M$ is perfect.
\item
$M$ is quasi-isomorphic to a finite semifree DG $B$-module.
\item
$\hh M$ is bounded below and degreewise finite, and $\Tor {}BMk$ is bounded.
\end{enumerate}
 \end{proposition}
 
 \begin{proof}
The definition yields (ii)$\implies$(i).  For (i)$\implies$(iii), since the conclusions in (iii) are inherited by direct summands, we may assume $M$ is finite semifree; then $\Tor {}BMk$ is isomorphic to $\hh{M\otimes_Bk}$, and thus bounded, while induction on  $\rank_{\nat B}\nat M$, using that each $\HH iM$ is noetherian, shows that $\hh M$ is bounded below and degreewise finite.  For (iii)$\implies$(ii), let $F\xra{\simeq}M$ be a minimal semifree resolution and note that by \eqref{eq:minimalL} $F$ has a finite semifree basis.
  \end{proof}

\begin{proof}[Proof of Theorem~\emph{\ref{thm:dgfriendly2}}]
Let $B'$ be the Koszul extension of $B$ offered by the hypothesis, and set $M'=B'\otimes_BM$ and $N'=B'\otimes_BN$.  
By parts (1) and (4) of Lemma~\ref{lem:exterior}, the $\HH 0{B'}$-modules $\hh{M'}$ and $\hh{N'}$ are bounded and 
degreewise finite, and $\Tor {}{B'}{M'}{N'}$ is bounded, by part (2) of that lemma.  By Lemma~\ref{lem:exterior}(5), $B'$ is a local DG algebra with residue field $k$. Since $B'$ is quasi-isomorphic to $A\ltimes {}^{\varepsilon}W$, it follows that $\hh W$ is degreewise finite and $\HH {< 0}W=0$. As $\hh W$ is nonzero and bounded, by hypothesis, Theorem~\ref{thm:dgfriendly} applies and yields that one of the DG $B'$-modules $M'$ and $N'$ is perfect; assume that the first one is.

The inclusion $B\subseteq B'$ commutes with the canonical augmentations to $k$.  Thus, Lemma~\ref{lem:exterior}(3) yields $\Tor {}BMk \cong \Tor {}{B'}{M'}k$.  Recalling that $M'$ is perfect over $B'$, we conclude that $M$ is perfect over $B$ by referring, twice, to Proposition~\ref{pr:perfect}.
\end{proof}

\section{Local rings}

We say that $(R,\fm,k)$ is a \emph{local ring} if $R$ is commutative noetherian ring with unique 
maximal ideal $\fm$, and $k=R/\fm$ is the residue field.  Let $e$ denote the minimal number of 
generators of $\fm$, and recall that $e-\depth R$ is non-negative.  We fix some minimal 
generating set of $\fm$ and let $K^R$ denote the Koszul complex on this set.

Clearly, local rings are precisely those  local DG algebras, in the sense of \ref{ch:localDGA}, 
which are zero in non-zero degrees.  In particular, the results of the preceding section apply 
directly to complexes over local rings with finitely generated homology.  Note that a perfect
DG $R$-module is simply one that is quasi-isomorphic to a bounded complex of finite free $R$-modules.

As a first application, we recover some known results about modules over local rings.
Formula \eqref{eq:PS} specializes to the following result of Herzog \cite[Theorem 1]{He}:
 
   \begin{proposition}
    \label{prop:Herzog}
If $(R,\fm,k)$ and $(S,\fn,k)$ are local rings, and $\alpha\colon S\to R$ and $\beta\colon R\to S$ are 
homomorphisms of rings, such that $\beta\alpha=\idmap^S$, then for every finite $S$-module $N$ 
there is an equality of formal power series
\begin{xxalignat}{3}
&\phantom{square}
&\poin RN &= \poin RS\poin SN\,.
&&\square
\end{xxalignat}
  \end{proposition}

Among the original characterizations of Golod rings, which appear in the next result, is 
the property that Massey products are defined for every finite set of elements of 
$\HH{\ges1}{K^R}$:  This is one direction of Golod's theorem in \cite{Go}.

  \begin{proposition}
    \label{prop:Golod}
Let $(R,\fm,k)$ be a local ring satisfying one of the conditions
  \begin{enumerate}[\rm\quad(a)]
    \item
$R$ is Golod; or
    \item
$R\cong S\ltimes k$ for some local ring $(S,\fn,k)$.
  \end{enumerate}

If $M$ and $N$ are finite $R$-modules and $\Tor{}RMN$ is bounded, then $M$ or $N$ has
finite projective dimension.
  \end{proposition}

  \begin{Remark}
Case (b) of the proposition is due to Nasseh and Yoshino, \cite[3.1]{NY}

In case (a), the conclusion is evident when $e=\edim R$, as then $R$ is regular.  If $e=\depth R+1$, 
then $R$ is a hypersurface ring, and the result is due to Huneke and Wiegand \cite[1.9]{HW}.  
For $e\ge\depth R+2$ the result is proved by Jorgensen \cite[3.1]{Jo}.  

Each one of those theorems required a different proof.
  \end{Remark}
  
  \begin{proof}
In case (b) the conclusion follows directly from Theorem \ref{thm:dgfriendly}.

It is proved in \cite[2.3]{Av} that all Massey products exist if and only $K^R\simeq k\ltimes W$ holds with some graded $k$-vector space $W$. We may assume $R$ is not regular, so that $W$ is nonzero. As $K^R$ is a Koszul extension of~$R$, Theorem \ref{thm:dgfriendly2} applies and shows that $M$ or $N$ is quasi-isomorphic to a bounded complex of free $R$-modules; that is, $\pdim_RM$ or $\pdim_RN$ is finite.
  \end{proof}
 
The value for local rings of the general form of Theorem \ref{thm:dgfriendly2} is demonstrated by the proof
of the next theorem, on which much of our work in \cite{AINW} depends.

As usual $\wh R$ denotes the $\fm$-adic completion of $R$.  Recall that Cohen's Structure Theorem yields an isomorphism
$\wh R\cong P/I$ for some regular local ring $(P,\fp,k)$ and ideal $I$ contained in $\fp^{2}$; any such isomorphism
is called a \emph{minimal Cohen presentation} of $\wh R$.

\begin{theorem}
\label{th:local}
Let $R$ be a local ring. Assume there exists a minimal Cohen presentation $\wh R\cong P/I$ satisfying
  \begin{enumerate}[\quad\rm(a)]
   \item
some minimal free resolution of $\wh R$ over $P$ has a structure of DG algebra; and
   \item
the $k$-algebra $B=\Tor{}P{\wh R}k$ is isomorphic to the trivial extension $A\ltimes W$ of a 
graded $k$-algebra $A$ by a graded $A$-module $W\ne0$ with $A_{\ges 1}\cdot W=0$.
  \end{enumerate}

If $M$ and $N$ are finite $R$-modules and $\Tor{}RMN$ is bounded, then $M$ or $N$ 
has finite projective dimension.
  \end{theorem}

\begin{proof}
In view of the faithful flatness of completions, the canonical isomorphisms
   \begin{align*}
\Tor{}{\wh R}{\wh R \otimes_RM}{\wh R\otimes_RN}&\cong\wh R\otimes_R\Tor{}RMN
   \\
\Tor{}{\wh R}{\wh R\otimes_RM}k&\cong\wh R\otimes_R\Tor{}RMk\cong\Tor{}RMk
   \end{align*}
show that we may assume that $R$ is complete, and hence $R\cong P/I$.

Let $K^P$ denote the Koszul complex on a minimal set of generators of $\fp$. It is a local DG algebra, in the sense of 
\ref{ch:localDGA}, and as $P$ is regular it has $\HH{}K\cong k$.

By (a), there is a DG $P$-algebra $B$, semifree as a DG $P$-module, with $\HH{}B=R$ and  $\partial(B)\subseteq\fp B$.  These properties yield the equality and the last isomorphism in the following string
\[
K^R\cong R\otimes_{P}K^P  \xleftarrow{\ \simeq\ } B\otimes_{P} K^P \xra{\ \simeq\ } B \otimes_{P} k =\HH{}{B \otimes_{P} k} \cong \Tor{}PRk
\]
of morphisms of DG algebras.  The quasi-isomorphisms are obtained by tensoring the augmentations $B\xra{\simeq} R$ and 
$K^P\xra{\simeq}k$ with  the bounded complexes of free $P$-modules $K$ and $B$, respectively. Due to (b), we get 
$K^R\simeq A\ltimes W$.  

As $K^R$ is a Koszul extension of $R$, Theorem~\ref{thm:dgfriendly2} yields the desired conclusion.
  \end{proof}

\appendix

\section{On differential graded modules}
\label{On differential graded modules}

This section is a collection of basic facts concerning DG modules over DG algebras, used in the body of the article. 
In most cases, further details (occasionally stated using slightly different terminology) can be found in \cite[Section~1]{AFL}.

\begin{chunk}
  \label{ch:DG}
Let $B$ be a DG algebra and $M$ a DG $B$-module.  Both are $\BZ$-graded and all their elements 
are homogenous.  We say that $M$ is \emph{bounded below} if $M_i=0$ for $i\ll0$, \emph{bounded}
if $M_i=0$ for $|i|\gg0$, and \emph{non-negative} if $M_i=0$ for $i<0$. Set
\[
\inf M:=\inf\{ i\mid M_{i}\ne 0\}\quad\text{and}\quad \sup M:=\sup\{ i\mid M_{i}\ne 0\}\,.
\]

We write $\shift M$ for the left DG $B$-module with $M_{n-1}$ as component of degree $n$, 
$\partial^{\shift M} (m) = - \partial^{M} m$, and $B$ acting by $b\cdot m = (-1)^{|b|}bm$, where $|b|$ is the degree of $b$.

The homology $\hh M$ is a graded module over the graded ring $\hh B$.  In particular, $\HH 0B$ is a ring and 
each $\HH iM$ is a left $\HH 0B$-module.  When these modules are finite for all  $i\in\BZ$, we say 
that $\hh M$ is \emph{degreewise finite}.
Morphisms of DG objects that induce isomorphisms in homology are called \emph{quasi-isomorphisms}.

The $\BZ$-graded ring underlying $B$ is denoted by $\nat B$, and $\nat M$ denotes the $\BZ$-graded left 
$\nat B$-module underlying $M$. 
 \end{chunk}

\begin{chunk}
  \label{ch:semifree}
Let $F$ be left DG $B$-module.  A \emph{semibasis} of $F$ is a well-ordered subset $\{\bsf\}$ of $F$, which is a
basis of $\nat F$ over $\nat R$ and satisfies $d(f)\in\sum_{e<f}Be$ for each $f$ in $\bsf$.  

A DG $B$-module that has a (finite) semibasis is said to be (\emph{finite}) \emph{semifree}.

A DG $B$-module that is quasi-isomorphic to a direct summand of some finite semifree DG $B$-module is called 
\emph{perfect}.
 \end{chunk}

\begin{chunk}
\label{ch:tor}
Each left DG $B$-module $M$ admits a \emph{semifree resolution}; that is, a quasi-isomorphism of left DG $B$-modules $F\to M$ 
with $F$ semifree; see \cite[\S1.11]{AH}.  After choosing a resolution for each $M$, for every right DG $B$-module $L$ one sets
\[
\Tor {}BLM = \hh{L\otimes_{B}F}\,.
\]
The result is independent of the choice of semifree resolutions; see \cite[Remark 1.14]{AH}.
  \end{chunk}

\begin{chunk}
\label{ch:semiflat}
A left DG $B$-module $F$ is said to be \emph{semiflat} if the functor $(?\otimes_{B}F)$ preserves injective quasi-isomorphisms of right DG $B$-modules; equivalently, $(?\otimes_{B}F)$ preserves quasi-isomorphisms and the graded $\nat B$-module $\nat F$ is flat. 

If $F\to G$ is a quasi-isomorphism of semiflat left DG $B$-modules, then the induced map $L\otimes_{B}F\to L\otimes_{B}G$ is a quasi-isomorphism
for every right DG $B$-module $L$. 

Semifree DG modules are semiflat. If $F\xra{\simeq} M$ is a quasi-isomorphism with $F$ semiflat, then for any right DG $B$-module $L$, there is an isomorphism
\[
\Tor {}BLM \cong \hh{L\otimes_{B}F}\,.
\]
\end{chunk}

\begin{chunk}
\label{ch:flat-maps}
Let $\beta\colon B\to C$ be a morphism of DG algebras. 

If $M$ is a semiflat left DG $B$-module, then $C\otimes_{B}M$ is a semiflat left DG $C$-module.

If ${}^{\beta}C$ is semiflat and $N$ is a semiflat  left DG $C$-module, then  ${}^{\beta}N$ is semiflat. 
\end{chunk}

\begin{chunk}
\label{ch:tor2} 
Let $\beta\colon B\to C$ be a morphism of DG algebras.  Let $L$ and $L'$ be right DG modules and 
$M$ and $M'$ be a left DG modules, over $B$ and $C$, respectively.  

Morphisms of complexes 
$\lambda\colon L\to L'$ and $\mu\colon M\to M'$ are called $\beta$-\emph{equivariant} if 
\[
\lambda(l b) = \lambda(l)\beta(b) \quad\text{and}\quad \mu(bm) = \beta(b)\mu(m)
\]
hold for all $b\in B$, $l\in L$ and $m\in M$.  Such maps define a natural homomorphism 
\[
\Tor{}{\beta}{\lambda}{\mu}\colon \Tor{}BLM \lra \Tor{}{C}{L'}{M'}\,.
\]
of graded abelian groups.  It is bijective if $\hh\beta$, $\hh\lambda$, and $\hh\mu$ are; see \cite[1.5]{AFL}.
\end{chunk}

   \begin{chunk}
      \label{ch:equivalence}
Two DG algebras $B$ and $C$ are said to be \emph{quasi-isomorphic} if there exists a chain $\mathsf{f}$ of quasi-isomorphisms 
of DG algebras linking $B$ and $C$.

Such a chain $\mathsf{f}$ yields an isomorphism $\mathsf{f}_*\colon \hh{B}\xra{\cong}\hh{C}$ of graded rings.  To each right 
DG $B$-module $L$ and left DG $B$-module $M$ it assigns a right DG $C$-module $\mathsf{f}L$, a left DG $C$-module 
$\mathsf{f}M$,  isomorphisms $\hh{L}\xra{\cong}\hh{\mathsf{f}L}$ and $\hh{M}\xra{\cong}\hh{\mathsf{f}M}$ that are
$\mathsf{f}_*$-equivariant, and an isomorphism $\Tor{}BLM\xra{\cong}\Tor{}{C}{\mathsf{f}L}{\mathsf{f}M}$.

In addition, $M$ is perfect over $B$ if and only if $\mathsf{f}M$ is perfect over $C$.

These statements reflect various properties of a triangle equivalence, induced by~$\mathsf{f}$, 
of the derived categories of DG $B$-modules and DG $C$-modules; see \cite[3.6.2]{ABIM}.
  \end{chunk}

\begin{chunk}
A DG algebra $B$ is \emph{graded-commutative} if all $b$, $b'$ in $B$ satisfy
\[
b\cdot b' = (-1)^{|b||b'|}b'\cdot b \quad \text{and}\quad \text{$b^{2}=0$ when $|b|$ is odd}
\]
Every right DG $B$-module $M$ then is canonically a left DG $B$-module, with action 
\[
b\cdot m := (-1)^{|b||m|}mb\quad\text{for $b\in B$ and $m\in M$}\,,
\]
so when speaking of DG $B$-modules, we drop references to `left' or `right'; in particular, 
this refers to semifreeness and semiflatness. When $L$ and $M$ are DG $B$-modules,
$\Tor{}BLM$ is a graded $\hh B$-module and there is an $\hh B$-linear isomorphism 
\[
\Tor{}BLM \cong \Tor{}BML\,. 
\]
  \end{chunk}

We record a basic fact on the existence of resolutions that are also DG algebras.

\begin{chunk}
\label{ch:resolutions-map}
Each morphism $\beta\colon B\to C$ of graded-commutative DG algebras can be factored as
\[
B\xra{\ \iota\ } \wt C \overset{\ \epsilon\ }\twoheadrightarrow C
\]
where $\iota$ and $\epsilon$ are morphisms of DG algebras, $\iota$ is injective, $\epsilon$ is a surjective quasi-isomorphism,
and the DG $B$-module $\Coker(\iota)$ (hence also $\wt C$) is semiflat.

Any such factorization is called a \emph{semiflat DG algebra resolution of $\beta$}.
  \end{chunk}

\end{document}